\documentclass[12pt]{amsart}
\usepackage{graphicx}
\usepackage[font=small,labelfont=bf]{caption}
\newtheorem{theorem}{Theorem}[section]

\newtheorem{example}{Example}[section]

\newtheorem{prop}{Proposition}[section]
\newtheorem{remark}{Remark}[section]
\newtheorem{defi}{Definition}[section]

\def\O{\Omega}
\def\R{{\mathbb R}}

\def\g{\gamma}

\begin{document}
\section*{}

\setcounter{equation}{0}
\title[]{On the interpolation space $(L^p(\O), W^{1,p}(\O))_{s,p}$ in non-smooth domains}

\author{Irene Drelichman}
\address{IMAS (UBA-CONICET) \\
Facultad de Ciencias Exactas y Naturales\\
Universidad de Buenos Aires\\
Ciudad Universitaria\\
1428 Buenos Aires\\
Argentina}
\email{irene@drelichman.com}

\author{Ricardo G. Dur\'an}
\address{IMAS (UBA-CONICET) and Departamento de Matem\'atica\\
Facultad de Ciencias Exactas y Naturales\\
Universidad de Buenos Aires\\
Ciudad Universitaria\\
1428 Buenos Aires\\
Argentina}
\email{rduran@dm.uba.ar}

\thanks{Supported by ANPCyT under grant PICT 2014-1771, by CONICET under grant 11220130100006CO  and by Universidad de Buenos Aires under grant 20020120100050BA. The authors are members of
CONICET, Argentina.}

\keywords{}


\begin{abstract}
We show that, for certain non-smooth bounded domains $\O\subset\R^n$, the real interpolation space  $(L^p(\O), W^{1,p}(\O))_{s,p}$ is  the subspace  $\widetilde W^{s,p}(\O) \subset L^p(\O)$ induced by the restricted fractional seminorm
$$
|f|_{\widetilde W^{s,p}(\O)} =  \Big( \int_\O \int_{|x-y|<\frac{d(x)}2} \frac{|f(x)-f(y)|^p}{|x-y|^{n+sp}} \, dy \,dx \Big)^\frac{1}{p}.
$$
In particular, the above result includes simply connected uniform domains in the plane, for which a characterization of the interpolation space was previously unknown.
\end{abstract}
\maketitle

\section{Introduction}

The purpose of this article is to   characterize the real interpolation space $(L^p(\O), W^{1,p}(\O))_{s,p}$ (see Definition \ref{interpolado}) for certain non-smooth  bounded domains $\O\subset\R^n$.

To be precise, let us recall that the usual fractional Sobolev space $W^{s,p}(\O)$ is the subspace of $L^p(\O)$ induced by the seminorm
$$
|f|_{W^{s,p}(\O)} =  \Big( \int_\O \int_\O \frac{|f(x)-f(y)|^p}{|x-y|^{n+sp}} \, dy \,dx \Big)^\frac{1}{p}.
$$
When $\O$ is a Lipschitz domain  it is known that $W^{s,p}(\O)$ coincides with the real interpolation space  $(L^p(\Omega), W^{1,p}(\Omega))_{s,p}$ (see the discussion below). However,  it is also known that this cannot be the case for arbitrary domains, since one has $W^{1,p}(\O)\subset (L^p(\O), W^{1,p}(\O))_{s,p} \subset L^p(\O)$, and it is easy to construct  domains for which $W^{1,p}(\O) \not\subset W^{s,p}(\O)$ for certain values of $s$, a typical example being a square minus a slit (see Example \ref{slit}).

Our main result is that, for a class of domains in $\R^n$ which we call {\it admissible} (see Definition \ref{admissible}), which contains certain non-Lipschitz domains including simply connected uniform domains in the plane, there holds $(L^p(\O), W^{1,p}(\O))_{s,p}= \widetilde W^{s,p}(\O)$, where $\widetilde W^{s,p}(\Omega)$ is the subspace of $L^p(\O)$ induced by the seminorm
$$
|f|_{\widetilde W^{s,p}(\O)} =  \Big( \int_\O \int_{|x-y|<\frac{d(x)}2} \frac{|f(x)-f(y)|^p}{|x-y|^{n+sp}} \, dy \,dx \Big)^\frac{1}{p}.
$$
This larger fractional space has been previously introduced in the literature in connection with fractional Poincar\'e inequalities in irregular domains \cite{DD-frac, D, HV}, and it is known that $W^{s,p}(\O)=\widetilde W^{s,p}(\O)$ when $\O\subset \R^n$ is a Lipschitz domain \cite[equation (13)]{Dy}. 

The rest of the paper is organized as follows: in Section 2 we introduce some notations and necessary preliminaries; in Section 3 we define the class of admissible domains  and prove that if $\O\subset\R^n$ is admissible, then $(L^p(\O), W^{1,p}(\O))_{s,p}=\widetilde W^{s,p}(\O)$; finally, in Section 4 we consider some concrete examples of admissible domains, which include, among others, simply connected uniform domains in the plane.

\section{Notation and preliminaries}

Throughout the paper, $1\le p<\infty$ and $p'$ is its conjugate exponent, $\frac{1}{p}+\frac{1}{p'}=1$, $\O\subset \R^n$ ($n\ge 2$) will be a bounded domain, and $d(x)$ will denote the distance of a point $x\in\O$ to the boundary of $\O$. If $\O=\cup_j \O_j$, the distance of a point $x\in \O_j$ to the boundary of $\O_j$ will be denoted by $d_j(x).$ Finally, $C$  will denote a positive constant that may change even within a single string of inequalities.

\begin{defi}
\label{interpolado}
For $0<s<1$, the  interpolation space obtained by the real method (see, e.g. \cite{BS}) is given by
\begin{equation}
(L^p(\Omega), W^{1,p}(\Omega))_{s,p} = \{f: f \in L^p(\Omega) + W^{1,p}(\Omega) \mbox{ s.t. } |f|_{(L^p(\Omega), W^{1,p}(\Omega))_{s,p}} < \infty\},
\end{equation}
where
\begin{equation}
\label{semiK}
|f|_{(L^p(\Omega), W^{1,p}(\Omega))_{s,p}} = \Big\{  \int_0^\infty \Big(\ell^{-s} K(\ell,f)\Big)^p \, \frac{d\ell}{\ell} \Big\}^\frac{1}{p}
\end{equation}
and the $K$-functional is given by
\begin{equation}
\label{defiK}
K(\ell, f) = K(\ell, f ;L^p(\O), W^{1,p}(\O)) = \inf\{ \|g\|_{L^p(\O)} + \ell \|h\|_{W^{1,p}(\O)} : f = g + h \}
\end{equation}
\end{defi}

\begin{remark}
\label{cotaK}
Observe that it is equivalent to consider
the integral in \eqref{semiK} in the interval $(0,\tau)$ for some sufficiently small $\tau>0$. Indeed, since $K(\ell, f)\le \|f\|_{L^p}$, we always have
$$
\int_\tau^\infty \Big(\ell^{-s} K(\ell,f)\Big)^p \, \frac{d\ell}{\ell} \le \|f\|_{L^p} \int_\tau^\infty  \ell^{-sp} \, \frac{d\ell}{\ell} < +\infty.
$$
Also, it suffices to consider, for a given decomposition, $\int_0^\tau \ell^{-sp} (\|g\|_{L^p} + \ell \|\nabla h\|_{L^p})^p \, \frac{d\ell}{\ell}$. This is because the remaining term satisfies
\begin{align*}
\int_0^\tau (\ell^{1-s} \|h\|_{L^p})^p \, \frac{d\ell}{\ell} &\le C \left\{\int_0^\tau (\ell^{1-s} \|f\|_{L^p})^p \, \frac{d\ell}{\ell} +  \int_0^\tau  (\ell^{1-s} \|g\|_{L^p})^p \, \frac{d\ell}{\ell}\right\}\\
& \le C  \left\{ \|f\|_{L^p}^p \int_0^\tau \ell^{p-sp} \, \frac{d\ell}{\ell} + \int_0^\tau (\ell^{-s} \|g\|_{L^p})^p \, \frac{d\ell}{\ell}\right\}
\end{align*}
and, therefore, it will be bounded provided the other terms are.
\end{remark}

It is known that when $\O\subset \R^n$ is a Lipschitz domain,  $(L^p(\O), W^{1,p}(\O))_{s,p}=B^s_{p,p}(\O)$ (see \cite{JS} and \cite[Section 11]{DS-book}), where $B^s_{p,p}$ is a Besov space, we refer the reader to \cite{DS} for its definition and properties. But, for Lipschitz domains one also has   $B^s_{p,p}(\O)=W^{s,p}(\O)$ (see \cite[Theorem 6.7]{DS}) and, hence, $(L^p(\O), W^{1,p}(\O))_{s,p}= W^{s,p}(\O)$ as mentioned in the introduction.  For more general domains, namely $(\varepsilon, \delta)$-uniform domains, it is also possible to interpolate between $L^p(\O)$ and spaces in the Besov scale (see \cite{DS}),   but to our knowledge there are no results proving or disproving that $B^s_{p,p}(\O)=W^{s,p}(\O)$ for such domains, and hence no known characterizations of $(L^p(\O), W^{1,p}(\O))_{s,p}$. 

We first notice that, for arbitrary domains, the interpolation space may be larger than $W^{s,p}(\O)$, as the next example shows. 

\begin{example}
\label{slit}
Let $\O=(-1,1)^2 \setminus ((0,1)\times \{0\})$ and let $f\in W^{1,p}(\O)$ such that $f(x)=1$ for $x\in (\frac12, 1)\times (0,1)$, and $f(x)=0$ for $x\in (\frac12, 1)\times (-1,0)$. Then $f\not\in W^{s,p}(\O)$ for $s>\frac{1}{p}$. On the other hand, it follows from the definition of the interpolation space that $W^{1,p}(\O) \subset (L^p(\Omega), W^{1,p}(\Omega))_{s,p}$, whence, $(L^p(\O), W^{1,p}(\O))_{s,p} \neq W^{s,p}(\O)$ for these values of $s$.
\end{example}

For our proof we will require the existence of  partitions of unity supported on John domains. These domains were introduced by F. John in \cite{J} and  given that name in \cite{MS}, we recall their definition below:
\begin{defi}
A bounded domain $\O\subset\R^n$ is a John domain if there exists $x_0\in\O$, a family of
rectifiable curves given by $\g(t,y)$, $0\le t\le 1$, $y\in\O$,
and  positive constants $\lambda$ and $k$ such that,
\begin{enumerate}
\item $\gamma(0,y)=y$, $\gamma(1,y)=x_0$
\medskip
\item  $d(\gamma(t,y))\ge \lambda t$  for all $t\in[0,1]$\label{prop2}
\medskip
\item  $|\dot\gamma(t,y))|\le k$
\end{enumerate}
\end{defi}

\section{Proof of our main result}

\begin{theorem}
\label{teo1}
For any bounded domain $\O \subset \R^n$, there holds $(L^p(\Omega), W^{1,p}(\Omega))_{s,p} \subset \widetilde{W}^{s,p}(\Omega)$.
\end{theorem}

\begin{proof}
Given $f\in (L^p(\Omega), W^{1,p}(\Omega))_{s,p}$, observe that, since  $L^p(\Omega) + W^{1,p}(\Omega) = L^p(\Omega)$,   $\|f\|_{L^p} \le C |f|_{(L^p, W^{1,p})_{s,p}}$. Hence, it  suffices to prove that 
\begin{equation}
\int_{\Omega} \int_{|x-y|< \frac{d(x)}{2}} \frac{|f(x)-f(y)|^p}{|x-y|^{n+sp}} \, dy \, dx \le C \int_0^\infty \Big(\ell^{-s} K(\ell,f)\Big)^p \, \frac{d\ell}{\ell} \end{equation}

To this end, for each $r\in \R_+$  choose $g_r$ and $h_r$ such that $f=g_r+h_r$ and $\|g_r\|_{L^p}+ r \|\nabla h_r\|_{L^p}\le 2 K(f,r)$. By density, we may also assume that $h_r$ is smooth.

We clearly have that $|f|_{\widetilde W^{s,p}(\Omega)} \le  C ( |h_r|_{\widetilde W^{s,p}(\Omega)} + |g_r|_{\widetilde W^{s,p}(\Omega)})$ for every $r\in\R_+$ and we may bound both terms separately.

 Observe that $x+tz \in \Omega$ for all $t\in [0,1]$ and all $x,z \in \Omega$ such that $|z|<\frac{d(x)}{2}$. Therefore, extending $\nabla h_r$ by zero outside $\O$, we have
 \begin{align*}
 \int_{\Omega} \int_{|z|\le \frac{d(x)}{2}} \frac{|h_r(z+x)-h_r(x)|^p}{|z|^{n+sp}} \, dz \, dx &\le \int_\Omega  \int_{|z|\le \frac{d(x)}{2}} \frac{1}{|z|^{n+sp}} \Big| \int_0^1 \nabla h_r(x+tz) \cdot z \, dt \Big|^p \, dx \, dz\\
&\le \int_{\mathbb{R}^n}   \frac{|z|^p}{|z|^{n+sp}} \int_0^1  \int_{\mathbb{R}^n}  |\nabla h_r(x+tz)|^p \, dx  \, dt \, dz\\
 &\le \int_{\R^n}  \frac{|z|^p}{|z|^{n+sp}} \|\nabla h_r\|_p^p \, dz \\
 \end{align*}

 Similarly, extending $g_r$ by zero outside $\O$, we have
 \begin{align*}
 \int_\Omega  \int_{|z|\le \frac{d(x)}{2}} \frac{|g_r(z+x)-g_r(x)|^p}{|z|^{n+sp}} \, dz \, dx &\le C \int_\Omega \int_\Omega \frac{|g_r(x)|^p}{|z|^{n+sp}} \, dx \, dz + \int_\Omega \int_{|z|<\frac{d(x)}{2}} \frac{|g_r(z+x)|^p}{|z|^{n+sp}} \, dz \, dx  \\
 &\le C\int_{\mathbb{R}^n} \frac{\|g_r\|_p^p}{|z|^{n+sp}} \, dz + \int_{\mathbb{R}^n} \int_\Omega \frac{|g_r(w)|^p}{|z|^{n+sp}} \, dw \, dz\\
 &\le C\int_{\mathbb{R}^n} \frac{\|g_r\|_p^p}{|z|^{n+sp}} \, dz 
 \end{align*}
 
 Finally, choosing $r=|z|$ and using polar coordinates, we obtain
 
 \begin{align*}
|f|_{\widetilde W^{s,p}(\Omega)}^p &\le C \Big( \int_0^\infty \frac{1}{r^{n+sp-p}} \|\nabla h_r\|_p^p \, r^{n-1} \, dr + \int_0^\infty \frac{\|g\|_p^p}{r^{n+sp}} \, r^{n-1} \, dr \Big) \\
 &\le  C \int_0^\infty r^{-sp} \Big(r \|\nabla h_r\|_{L^p} + \|g_r\|_{L^p}\Big)^p \, \frac{dr}{r} \\
 &\le C  \int_0^\infty \Big( r^{-s} K(r,f)\Big)^p \, \frac{dr}{r}.
 \end{align*}

This concludes the proof.
 \end{proof}
 
 \begin{defi}
 \label{admissible}
 We say that a bounded domain $\O \subset\R^n$ is  admissible  provided that there exists $\tau>0$ such that, for each $0<\ell<\tau$, there exist a  partition $\O = \cup_j \O_{\ell,j}$ and an associated partition of unity $\{\psi_{\ell,j}\}_j \in W^{1,\infty}(\O_{\ell,j})$ with the following properties:
 \begin{enumerate}
 \item  $\sum_j \chi_{\O_{\ell,j}}(x) \le C$ for all $x\in\R^n$
 \item $\|\psi_{\ell,j}\|_\infty \le C$, $\|\nabla\psi_{\ell,j}\|_\infty \le C \ell^{-1}$ for every $j$
 \item Each $\O_{\ell,j}$ is a John domain with constants of order $\ell$
  \item $diam(\O_{\ell,j}) \sim \ell$ for every $j$
 \end{enumerate}
 The reader should remark that each of these partitions is necessarily finite (depending on $\ell$) because the $\O_{\ell,j}$'s are a covering of a bounded domain made of finitely overlapping sets of essentially the same size.
 \end{defi}

\begin{theorem}
If $\O\subset\R^n$ is an admissible domain as in Definition \ref{admissible}, then $\widetilde{W}^{s,p}(\Omega) = (L^p(\Omega), W^{1,p}(\Omega))_{s,p}$.
\end{theorem}

\begin{proof}
Clearly, it suffices to prove that $\widetilde{W}^{s,p}(\Omega) \subset (L^p(\Omega), W^{1,p}(\Omega))_{s,p}$, since the converse is  always true by Theorem \ref{teo1}.

Given $f\in  \widetilde{W}^{s,p}(\O)$, for each $\ell$ sufficiently small, we let $h(y)=\sum_j f_j \psi_j(y)$, for certain values  $f_j$ that we will choose shortly, and $\psi_j$  the partition of unity given by Definition \ref{admissible}. Here, $f_j=f_{\ell,j}$ and $\psi_j=\psi_{\ell, j}$, but we have chosen to simplify notation. The reader should keep in mind that throughout this proof $f_j$ and $\psi_j$ depend on $\ell$.

To define $f_j$, recall that, by Definition \ref{admissible}, each $\O_j$ is a John domain with a distinguished point, say $x_j$, and constants of order $\ell$. Hence, 
there exists a John curve $\gamma_j$ such that $\gamma_j(0,y)=y$, $\gamma_j(1,y)=x_j$,  $d(\gamma_j(t,y))\ge \lambda \ell t$ and $|\dot \gamma_j|\le k\ell$. 

Observe that, if  $z\in B(x_j, \frac{\lambda \ell}{4})$ and we let $\tilde \gamma_j(t,y)=\gamma_j(t,y)+t(z-x_j)$, there holds   $\tilde \gamma_j(0,y)=y$ and $\tilde\gamma_j(1,y)=z$. Moreover, $|\gamma_j -\tilde\gamma_j|\le t|z-x_j| \le \frac{t\lambda \ell}{4}$, which implies that $\tilde\gamma_j \subset \Omega_j$.

Now, choose a smooth function $\varphi$ (also depending on $\ell$) such that  $supp(\varphi) \subset B(0,\frac{\lambda \ell}4)$, $ \varphi\ge 0$,  $\int \varphi=1$, $\|\varphi\|_p\le C \ell^{-\frac{n}{p'}}$, $\|\nabla\varphi\|_p\le C \ell^{-\frac{n}{p'}-1}$, and let $u(x,t)=f*\varphi_t(x)$ where $\varphi_t(x)= t^{-n}  \varphi(\frac{x}{t})$. We define  $f_j = \int u(z,1) \varphi(z-x_j) \,dz$.

Recall we want to estimate the $K$-functional, so it suffices to prove we can bound
\begin{equation}
\label{cotaK}
\int_0^\tau \ell^{-sp} \Big(\|g\|_{L^p(\O)} + \ell \|\nabla h\|_{L^p(\O)}\Big)^p \frac{d\ell}{\ell}\le C |f|_{\widetilde W^{s,p}(\O)}
\end{equation}
for  $\tau>0$ as in Definition \ref{admissible} (see Remark \ref{cotaK}).

By hypothesis, we have
\begin{equation}
\label{g}
\|g\|_{L^p(\O)}^p=\|f-h\|_{L^p(\Omega)}^p \le C \sum_j \|f-f_j\|_{L^p(\Omega_j)}^p,
\end{equation}
so, we begin by obtaining a pointwise bound for $f-f_j$. This computation is similar to the one used in \cite[Proposition 3.1]{DD-frac}. 

Observe that $\varphi(z-x_j)$ is supported in $B(x_j, \frac{\lambda\ell}{4})$ and has integral 1, hence, if we let $\eta_j(t)=u(\gamma_j(t,y)+t(z-x_j),t)$, we have that $\eta_j(0)=u(y,0)=f(y)$, $\eta_j(1)=u(z,1)$, and we may write
\begin{align*}
f(y)-f_j &= \int_{B(x_j, \frac{\lambda \ell}{4})} (f(y) -  u(z,1)) \varphi(z-x_j) \,dz \\
&= -\int_0^1  \int_{B(x_j, \frac{\lambda \ell}{4})} \eta_j'(t) \varphi(z-x_j) \, dz \, dt\\
&= -\int_0^1  \int_{B(x_j, \frac{\lambda \ell}{4})} \nabla u (\gamma_j(t,y)+t(z-x_j),t) \cdot (\dot\gamma_j +z-x_j)\varphi(z-x_j) \, dz \, dt \\
&\quad -\int_0^1  \int_{B(x_j, \frac{\lambda \ell}{4})}  \frac{\partial u}{\partial t} (\gamma_j(t,y)+t(z-x_j),t) \varphi(z-x_j) \, dz \, dt \\
\end{align*}

Making the change of variables $x=\tilde\gamma_j(t,y)= \gamma_j(t,y)+t(z-x_j)$, we know that $x\in \Omega_j$  and $dx= t^n dz$, whence,

\begin{align*}
f(y)-f_j &= -\int_0^1 \int_{\Omega_j} \nabla u(x,t) \cdot \Big(\dot\gamma_j + \frac{x-\gamma_j}{t}\Big) \varphi\Big(\frac{x-\gamma_j}{t}\Big) \frac{1}{t^n} \, dx \, dt\\
&\quad -\int_0^1  \int_{\Omega_j}  \frac{\partial u}{\partial t} (x,t) \varphi\Big(\frac{x-\gamma_j}{t}\Big)  \frac{1}{t^n}\, dx \, dt \\
\end{align*}

Now, using that  $\int \nabla \varphi=0$ and that $\int\frac{\partial\varphi_t}{\partial t}(x)dx=0$, we have that
$$
\nabla u(x,t) = f* \nabla \varphi_t(x) = \int_{\mathbb{R}^n} (f(x)-f(w)) \frac{1}{t^{n+1}} \nabla\varphi\Big(\frac{x-w}{t}\Big) \, dw,
$$
and
$$
\frac{\partial u}{\partial t}(x)= f*\frac{\partial }{\partial t} \varphi_t(x) = \int_{\mathbb{R}^n} (f(x)-f(w))  \Big[ \nabla\varphi\Big( \frac{x-w}{t} \Big) \cdot \Big(\frac{x-w}{t^{n+2}}\Big)+\varphi \Big(\frac{x-w}{t}\Big) \frac{1}{t^{n+1}}\Big] \, dw.
$$

Therefore, $f-f_j=-I_1-I_2-I_3$  with
\begin{align*}
I_1=& \int_0^1 \int_{\Omega_j}  \int_{\mathbb{R}^n} (f(x)-f(w)) \frac{1}{t^{n+1}} \nabla\varphi\Big(\frac{x-w}{t}\Big) \cdot \Big(\dot\gamma_j + \frac{x-\gamma_j}{t}\Big)  \, \varphi\Big(\frac{x-\gamma_j}{t}\Big) \frac{1}{t^n}  \, dw \, dx \, dt\\
I_2=&\int_0^1 \int_{\Omega_j}  \int_{\mathbb{R}^n} (f(x)-f(w)) \frac{1}{t^{n+1}}   \nabla\varphi\Big( \frac{x-w}{t} \Big) \cdot \Big(\frac{x-w}{t}\Big) \, \varphi\Big(\frac{x-\gamma_j}{t}\Big) \frac{1}{t^n}  \, dw \, dx \, dt\\\
I_3=&\int_0^1 \int_{\Omega_j}  \int_{\mathbb{R}^n} (f(x)-f(w)) \frac{1}{t^{n+1}} \varphi \Big(\frac{x-w}{t}\Big)    \, \varphi\Big(\frac{x-\gamma_j}{t}\Big) \frac{1}{t^n}  \, dw \, dx \, dt\\
\end{align*}

By hypothesis, $|\dot \gamma_j|<C\ell$, and using that $supp(\varphi) \subset B(0,\frac{\lambda \ell}4)$, we also have $|x-\gamma_j|<\frac{\lambda\ell t}4$ and $|x-w|\le \frac{\lambda \ell t}{4}$. Hence,
 
 \begin{align*}
|f-f_j| & \le C \int_0^1 \int_{\Omega_j}  \int_{\mathbb{R}^n} |f(x)-f(w)|  \frac{\ell}{t^{2n+1}} \Big|\nabla\varphi\Big(\frac{x-w}{t}\Big)\Big|  \, \Big|\varphi\Big(\frac{x-\gamma_j}{t}\Big)\Big|  \, dw \, dx \, dt\\
&+ C \int_0^1 \int_{\Omega_j}  \int_{\mathbb{R}^n} |f(x)-f(w)|  \frac{1}{t^{2n+1}} \Big|\varphi\Big(\frac{x-w}{t}\Big)\Big|  \, \Big|\varphi\Big(\frac{x-\gamma_j}{t}\Big)\Big|  \, dw \, dx \, dt\\
\end{align*}

Going back to \eqref{g}, we have
\begin{align*}
&\int_0^\tau \|g\|_{L^p(\O)}^p \, \ell^{-sp} \, \frac{d\ell}{\ell}\\
&\le C \int_0^\tau \sum_j \|f-f_j\|_{L^p(\O_j)}^p \, \ell^{-sp} \frac{d\ell}{\ell}\\
& \le C \int_0^\tau \sum_j \int_{\Omega_j} \Big( \int_0^1 \int_{\Omega_j} \int_{\mathbb{R}^n} |f(x)-f(w)| \frac{\ell}{t^{2n+1}}  \Big|\nabla\varphi\Big(\frac{x-w}{t}\Big)\Big|  \, \Big|\varphi\Big(\frac{x-\gamma_j}{t}\Big)\Big|  \, dw \, dx \, dt \Big)^p  dy \, \ell^{-sp} \, \frac{d\ell}{\ell}\\
& + C \int_0^\tau \sum_j \int_{\Omega_j} \Big( \int_0^1 \int_{\Omega_j}  \int_{\mathbb{R}^n} |f(x)-f(w)|  \frac{1}{t^{2n+1}} \Big|\varphi\Big(\frac{x-w}{t}\Big)\Big|  \, \Big|\varphi\Big(\frac{x-\gamma_j}{t}\Big)\Big|  \, dw \, dx \, dt \Big)^p  dy \, \ell^{-sp} \, \frac{d\ell}{\ell}\\
&= I + II 
\end{align*}

We will begin by estimating the integral $I$. To this end, observe first that 
$$
d_j(\gamma_j(t,y))\le |\gamma_j(t,y)-x|+d_j(x) \le \frac{\lambda \ell t}{4}+ d_j(x) \le \frac{d_j(\gamma_j(t,y))}{4} + d_j(x)
$$
and, therefore, using again the support of $\varphi$,
$$
|x-w|\le \frac{\lambda \ell t}{4}\le \frac{d_j(\gamma_j(t,y))}{4}\le \frac{d_j(x)}3 \le \frac{d(x)}3<\frac{d(x)}2,
$$
and 
$$
|x-y|\le |x-\gamma_j| + |\gamma_j - y|\le \frac{\lambda \ell t}{4} + k\ell t.
$$

Using these estimates and H\"older's inequality (in $dx \, dt$),
\begin{align*}
I &\le  C \int_0^\tau \sum_j \int_{\Omega_j} \int_0^1 \int_{|x-y|<C\ell t}  \Big( \int_{|x-w|<\frac{d(x)}{2}} |f(x)-f(w)|\frac{1}{t^{\frac{n+1}{p}+n+\frac{\varepsilon}{p'}}} \Big|\nabla\varphi\Big(\frac{x-w}{t}\Big)\Big| \, dw \Big)^p \, dx \, dt  \\
&\cdot \Big( \int_0^1 \int_{\R^n} \frac{1}{t^{n+1-\varepsilon}} \Big| \varphi\Big( \frac{x-\gamma_j}{t} \Big) \Big|^{p'}  \, dx \, dt  \Big)^\frac{p}{p'}   dy \, \ell^{-sp+p} \, \frac{d\ell}{\ell}
\end{align*}
In the integrals above, $\varepsilon>0$ is a sufficiently small exponent that will be chosen later. It allows us to compute separately
\begin{align*}
\Big( \int_0^1 \int_{\R^n} \frac{1}{t^{n+1-\varepsilon}} \Big| \varphi\Big( \frac{x-\gamma_j}{t} \Big) \Big|^{p'}  \, dx \, dt  \Big)^\frac{p}{p'}  &= \Big( \int_0^1 \frac{1}{t^{1-\varepsilon}} \|\varphi\|_{p'}^{p'} \, dt  \Big)^\frac{p}{p'} \\
&\le C \left( \int_0^1 \frac{1}{t^{1-\varepsilon}} \, \ell^{-\frac{np'}{p}}\, dt  \right)^\frac{p}{p'} \\
&=C \ell^{-n}
\end{align*}

Using this bound and H\"older's inequality again (in $dw$), we have
\begin{align*}
I &\le  C \int_0^\tau \sum_j \int_{\Omega_j} \int_0^1 \int_{|x-y|<C\ell t}  \Big( \int_{|x-w|<\frac{d(x)}{2}} |f(x)-f(w)|\frac{1}{t^{\frac{n+1}{p}+n+\frac{\varepsilon}{p'}}} \Big|\nabla\varphi\Big(\frac{x-w}{t}\Big)\Big| \, dw \Big)^p\\
&\qquad  \, dx \, dt \, dy \, \ell^{-n-sp+p} \, \frac{d\ell}{\ell}\\
&\le  C \int_0^\tau \sum_j \int_{\Omega_j} \int_0^1 \int_{|x-y|<C\ell t}  \Big( \int_{|x-w|<\frac{d(x)}{2}} |f(x)-f(w)|^p   \chi_{|x-w|<\frac{\lambda \ell t}4} \, dw  \Big) \\
&\quad \cdot \Big(\int_{\mathbb{R}^n}\frac{1}{t^n}  \Big|\nabla\varphi\Big(\frac{x-w}{t}\Big)\Big|^{p'} \, dw \Big)^\frac{p}{p'}  \frac{1}{t^{2n+1+\varepsilon(p-1)}}  \, dx \, dt \, dy \, \ell^{-n-sp+p} \, \frac{d\ell}{\ell}\\
&\le  C \int_0^\tau \sum_j \int_{\Omega_j} \int_0^1 \int_{|x-y|<C\ell t}   \int_{|x-w|<\frac{d(x)}{2}} |f(x)-f(w)|^p  \chi_{|x-w|<\frac{\lambda \ell t}4} \, dw   
 \frac{\ell^{-2n-sp}}{t^{2n+1+\varepsilon(p-1)}}   \, dx \, dt \, dy \, \, \frac{d\ell}{\ell}
\end{align*}
where in the last step we have used that  $\|\nabla\varphi\|_{L^{p'}(\R^n)}\le C \ell^{-\frac{n}{p}-1}$.

Hence, using the bounded overlap of the $\O_j$'s, we arrive at 
$$
I \le C \int_0^\tau \int_\O \int_0^1 \int_{|x-y|<C\ell t}   \int_{|x-w|<\frac{d(x)}{2}} |f(x)-f(w)|^p  \chi_{|x-w|<\frac{\lambda \ell t}4} \, dw   
 \frac{\ell^{-2n-sp}}{t^{2n+1+\varepsilon(p-1)}}   \, dx \, dt \, dy \, \, \frac{d\ell}{\ell}
$$

Notice that the above bound is independent of the partitions given by Definition \ref{admissible} at each scale $\ell$, so now we may interchange the order of integration and compute the integrals in the variables $y$, $\ell$ and $t$,  and we obtain
\begin{align*}
I &\le  C \int_\O  \int_0^1  \int_{|x-w|<\frac{d(x)}{2}} |f(x)-f(w)|^p \int_\frac{4|x-w|}{\lambda t}^\infty  \frac{\ell^{-2n-sp-1}}{t^{2n+1+\varepsilon(p-1)}} \int_{|x-y|<C \ell t}  \, dy \, d\ell \, dw  \, dt \, dx\\
&\le C  \int_\O  \int_0^1
\int_{|x-w|<\frac{d(x)}{2}} |f(x)-f(w)|^p
\int_\frac{4 |x-w|}{\lambda t}^\infty \ell^{-n-sp-1}   \, d\ell \, dw \frac{1}{t^{n+1+\varepsilon(p-1)}}  \, dt \, dx\\
&\le C \int_\O  \int_0^1  \int_{|x-w|<\frac{d(x)}{2}} \frac{|f(x)-f(w)|^p}{|x-w|^{n+sp}}  \, dw \, t^{sp-1-\varepsilon(p-1)}  \, dt \, dx\\
&\le C \int_\O    \int_{|x-w|<\frac{d(x)}{2}} \frac{|f(x)-f(w)|^p}{|x-w|^{n+sp}}  \, dw  \, dx\\
&\le C |f|_{\widetilde W^{s,p}(\Omega)}
\end{align*}
where, to integrate in $t$, we have used that $sp-1-\varepsilon(p-1)>-1$, which holds for sufficiently small $\varepsilon>0$.

It remains to bound the integral $II$, but the estimates are analogous if we observe that,  instead of using $\|\nabla\varphi\|_{L^{p'}(\R^n)}\le C \ell^{-\frac{n}{p}-1}$ we now have to use $\|\varphi\|_{L^{p'}(\R^n)}=\ell^{-\frac{n}{p}}$, which compensates for the missing $\ell$  in the numerator. This proves
\begin{equation*}
\int_0^\tau \|g\|_{L^p(\O)}^p \, \ell^{-sp} \, \frac{d\ell}{\ell} \le C |f|_{\widetilde W^{s,p}(\Omega)}.
\end{equation*}

We proceed now to bound
\begin{align*}
\int_0^\tau \|\nabla h\|_p^p \, \ell^{p(1-s)} \, \frac{d\ell}{\ell}.
 \end{align*}
Recall that, by definition of $h$, 
 $$
 |\nabla h(y)|=\Big|\sum_j f_j \nabla \psi_j(y)\Big| \le \sum_j |f_j-f(y)| |\nabla \psi_j(y)| \le \sum_j |f_j - f(y)| \frac{1}{\ell}
 $$
 Hence,
 \begin{equation*}
 \int_0^\tau \|\nabla h\|_{L^p(\O)}^p \, \ell^{p(1-s)} \frac{d\ell}{\ell} \le C \int_0^\tau \sum_j \|f_j - f\|_{L^p(\O_j)}^p \, \ell^{-sp} \frac{d\ell}{\ell} \le  C |f|_{\widetilde W^{s,p}(\Omega)}
 \end{equation*}
 as above. Therefore, putting both estimates together, we obtain  \eqref{cotaK}.
\end{proof}

\section{Examples }

In this section we show how one can prove that certain domains are admissible in the sense of Definition \ref{admissible}. 

Our first example includes the domain in Remark \ref{slit} (a typical John domain)  for which the interpolation space was previously uncharacterized. For simplicity, we consider domains in $\R^2$, but the reader will have no difficulty in generalizing the following proposition to higher dimensions.

We then turn to the case of simply connected uniform domains in the plane, and show that they are admissible. The proof relies on the fact that every  domain in this class is bi-Lipschitz equivalent to a member of a specific family of snowflake domains. 
 
\begin{prop}
If $\O$ is a connected open
\label{polygon} polygon in $\R^2$, possibly with interior holes or fractures (see Figure 1), then $(L^p(\O), W^{1,p}(\O))_{s,p}=\widetilde W^{s,p}(\O)$.
\end{prop}
\begin{proof}
To check that $\O$ is an admissible domain in the sense of Definition \ref{admissible} observe that, for each $\ell$ sufficiently small, $\O$ admits a triangulation $\mathcal{T}_\ell$ such that the radii of all inscribed circles and all circumcircles are comparable to $\ell$. 

We number all vertices of  $\mathcal{T}_\ell$  with the convention that each vertex shared by triangles separated by the boundary of the domain has to be considered as two separate vertices. Associated to the set of vertices $\{v_i\}_{i=1}^N$ (here $N=N(\ell)$), we can define sets  $\O_i$ as the interior of the union of all triangles of $\mathcal{T}_\ell$ containing $v_i$, and the Lagrange basis $\{\psi_i\}_{i=1}^N$ as the set of piecewise linear functions such that each $\psi_i$ is supported on $\bar\O_i$, $\psi_i(v_i)=1$, and $\psi_i(v_j)=0$ for all $j\neq i$. It is easy to check that this construction satisfies all the required hypotheses.
\end{proof}

\begin{center}
\includegraphics[width=4cm]{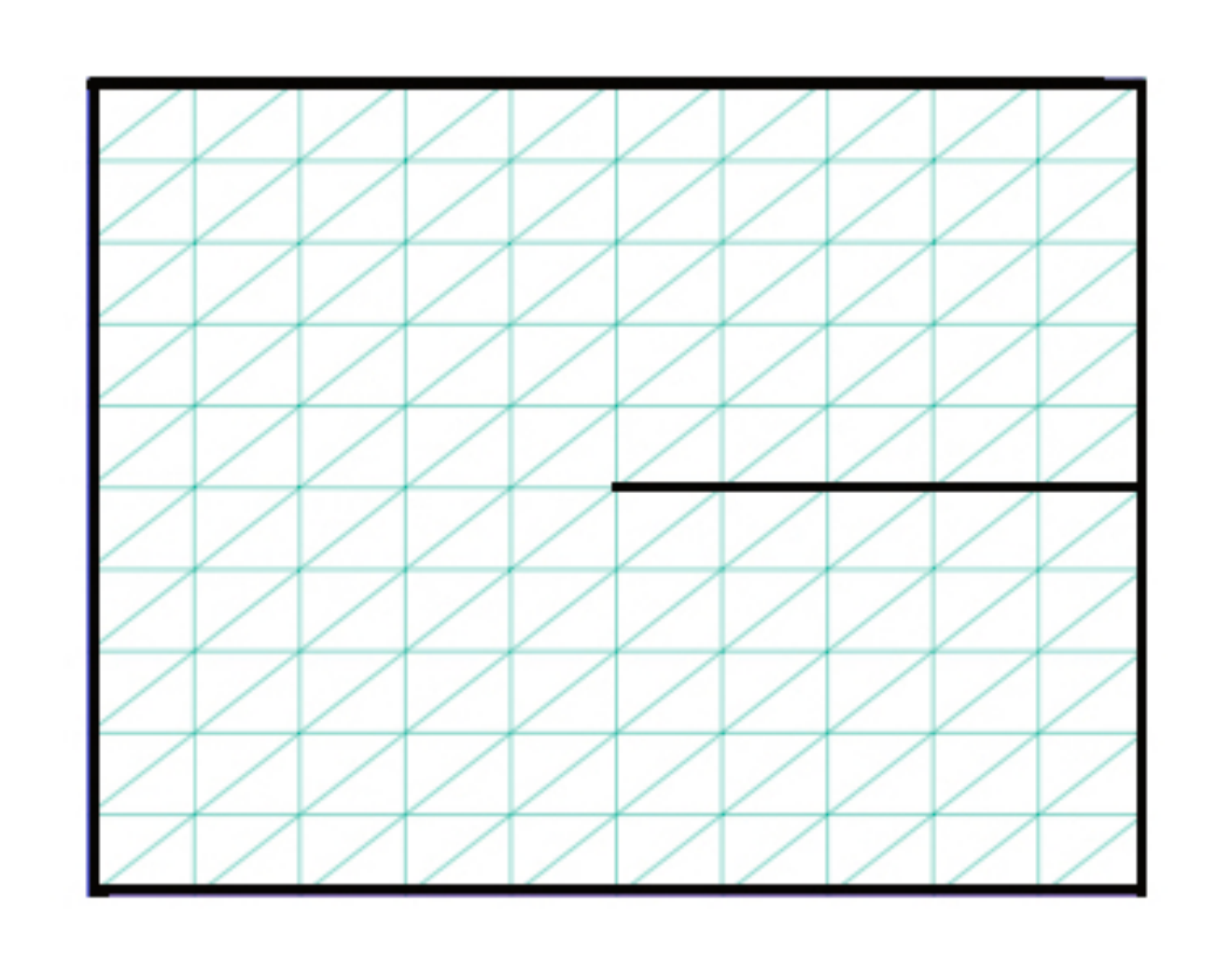} \; \; \; \; \;  \; \includegraphics[width=4cm]{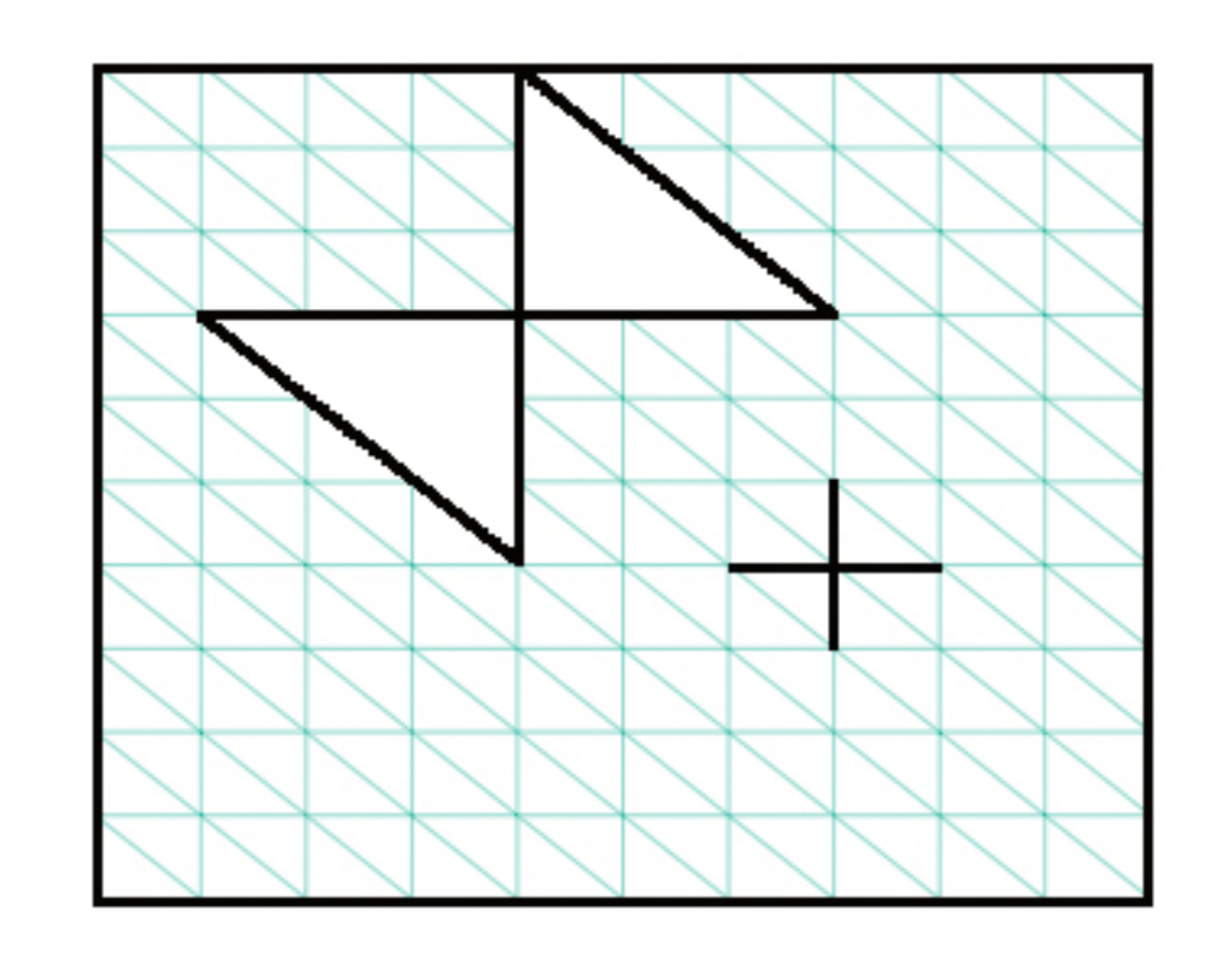}
\captionof{figure}{Examples of admissible non-Lipschitz domains in $\R^2$}
\end{center}

For our next example we need to recall some necessary definitions

\begin{defi}
We say that a bounded domain $\O\subset\R^n$ is uniform if there exist constants $a$ and $b$ such that each $x, y \in \O$ can be joined by an arc $\gamma$ in $\O$ with
\begin{enumerate}
\item $l(\gamma) \le a |x-y|$
\item $\min_{j=1,2} l(\gamma_j) \le b d(z)$ 
\end{enumerate}
for $z\in \gamma$, where $\gamma_1$, $\gamma_2$ are the components of $\gamma\setminus\{z\}$.
\end{defi}

\begin{remark}
Clearly, any uniform domain is a John domain, but the converse is not true, as can be seen in the case of the domain in Example \ref{slit}.
\end{remark}

\begin{defi}
We say that a homeomorphism $f:\R^2\to \R^2$ is quasiconformal if $f\in W^{1,2}_{loc}$ and there exists $K\ge 1$ such that $|Df(x)|^2\le K Jf(x)$ almost everywhere, where $Jf$ is the Jacobian. 
\end{defi}
\begin{defi}
 A quasicircle is the image of a circle under a quasiconformal map of the plane, and a quasidisk is the interior domain of a quasicircle.
\end{defi}

\begin{prop}
\label{uniform}
If $\O \subset \R^2$ is a simply connected uniform domain, then $(L^p(\O), W^{1,p}(\O))_{s,p}=\widetilde W^{s,p}(\O)$.
\end{prop}
\begin{proof}
It is known  that a simply connected planar domain $\O$ is a quasidisk if and only if it is a uniform domain \cite[Theorem 2.24]{MS}, and that  all quasidisks are  essentially snowflake domains, up to applying a bi-Lipschitz map of the plane \cite[Theorem 1.1]{R}. 

Let us sketch this equivalence: the author of \cite{R} constructs a family $\mathcal{S}$ of snowflake-type curves inductively, resembling the construction of the van Koch snowflake, but with two replacement options instead of one. Namely, he begins with the unit square and a fixed parameter $p$, $\frac14< p<\frac12$, and at each step he replaces every line segment (say of length $L$) in the $n$th generation polygon by either 
\begin{enumerate}
\item four disjoint subintervals of length $\frac{L}4$ or
\item a polygonal arc with four segments of length $pL$. 
\end{enumerate}
Then he proves that every closed quasicircle in the plane is the bi-Lipschitz image of some element  in $\mathcal{S}$. 

If $D$ is the bounded domain whose boundary is one of the snowflakes in $\mathcal{S}$ obtained by the above process, we can equivalently think of the construction of $D$ as starting with the unit square and choosing at each step whether to add or not rotated rescaled triangles. Therefore, every time we replace a line segment of the boundary by a polygonal arc, we may refer to the triangle delimited by the original boundary as a ``parent", and those added to its sides as its ``children". The ``children" of any subsequent step of a given triangle are its ``descendants".

To prove that a simply connected uniform domain $\O$ is admissible in the sense of Definition \ref{admissible}, assume that $\O=f(D)$ where $\partial D\in \mathcal{S}$ and $f:\R^2 \to \R^2$  is bi-Lipschitz with constant $K$, i.e. $K^{-1}|x-y|\le |f(x)-f(y)|\le K |x-y|$ for all $x,y\in\R^2$. 

For a given $\ell$ let $N$ be such that $p^N \sim \frac{\ell}{K}$, where $p$ is the parameter used to construct $D$. If we stop at the $N$-th step in the construction of $D$, we obtain a polygon $D_N \subseteq D$, which, as in the proof of Proposition \ref{polygon} admits a covering by finitely overlapping unions of triangles of size $\ell/K$, say $\{O_i\}$,  and an associated partition of unity $\{\phi_i\}$ of piecewise linear functions supported on those $O_i$'s. Now, if we define $\tilde O_i$ as the union of the triangles of $O_i$ and all their descendants, we have that $D=\cup_i \tilde O_i$, that the $\tilde O_i$'s are finitely overlapping, and that their diameters are of order $p^N \sim \frac{\ell}{K}$. Moreover, it is easy to see that  each $\tilde O_i$ is a John domain (actually, it is uniform) with constants of order $\ell$. 

To construct a  partition of unity on $D$, if $\phi_i$ vanishes on the boundary of $D_N$, we define  $\tilde\phi_i=\phi_i$ on $O_i$ and zero otherwise. If  $\phi_i$ does not vanish on the boundary of $D_N$, observe that, since it is  is piecewise linear, it has a natural extension $\tilde \phi_i$ (given by the same formula) to  $\tilde O_i$ and can be extended by zero otherwise. 

Notice that $\|\tilde \phi_i\|_\infty \le C$, $\|\nabla \tilde \phi_i\|_\infty \le C\ell^{-1}$, and $\sum_i \tilde \phi_i = 1$ on $D_N$,  but the latter may not be the case on $D \setminus D_N$, so we define $\tilde \psi_i = \frac{\tilde \phi_i}{\sum_j \tilde \phi_j}$ which clearly satisfies $\sum_i \tilde \psi_i =1$,  and also, using that $\sum_j \tilde \phi_j \ge 1$, 
$\|\tilde \psi_i\|_\infty \le C$ and  $\|\nabla \tilde \psi_i\|\le C \ell^{-1}$.

Finally, the required partition in $\O$ is given by $\O_i=f(\tilde O_i)$ and the associated partition of unity is given by $\psi_i =  \tilde \psi_i \circ f^{-1}$.
Since $f$ is bi-Lipschitz, each $\O_i$ is a John domain with has size and constants of order $\ell$ (see \cite[section 2.14]{MS})  and  it is easy to check that the remaining conditions of Definition \ref{admissible} are satisfied.
\end{proof}

\end{document}